\documentclass{amsart}

\newtheorem{lem}{Lemma}[section] 
\newtheorem{prp}[lem]{Proposition}
\newtheorem{thm}[lem]{Theorem} 
\newtheorem{crl}[lem]{Corollary} 
\newtheorem{defi}[lem]{Definition} 
\newtheorem{nota}[lem]{Notation}

\newtheorem{rque}[lem]{Remark}
\newtheorem{quest}[lem]{Question}

\numberwithin{equation}{section}

\newcommand\R{\mathbb{R}}

\newcommand\N{\mathbb{N}}
\newcommand\Z{\mathbb{Z}}

\newcommand\G{\mathbb{G}}

\usepackage{amsmath,amssymb,amsthm}
\usepackage{mathrsfs}
\usepackage{graphicx}
\usepackage{amsmath}
\usepackage[all]{xy}

\begin{document}

\title{Haagerup Approximation Property for quantum reflection groups}

\author{Fran\c cois Lemeux}
\address{Fran\c cois Lemeux, Laboratoire de math\'ematiques de Besan\c con, UFR Sciences et Techniques, Universit\' e de Franche-Comt\' e,16 route de Gray, 25000 Besan\c con, France}
\curraddr{}
\email{francois.lemeux@univ-fcomte.fr}
\thanks{}


\date{}

\dedicatory{}

\commby{Marius Junge}

\begin{abstract}
In this paper we prove that the duals of the quantum reflection groups $H_N^{s+}$ have the Haagerup property for all $N\ge4$ and $s\in[1,\infty)$. We use the canonical arrow $\pi: C(H_N^{s+})\to C(S_N^+)$ onto the quantum permutation groups, and we describe how the characters of $C(H_{N}^{s+})$ behave with respect to this morphism $\pi$ thanks to the description of the fusion rules binding irreducible corepresentations of $C(H_N^{s+})$ (\cite{BV09}). This allows us to construct states on the central $C^*$-algebra $C(H_N^{s+})_0$ generated by the characters of $C(H_{N}^{s+})$ and to use a fundamental theorem proved by M.Brannan giving a method to construct nets of trace-preserving, normal, unital and completely positive maps on the von Neumann algebra of a compact quantum group $\G$ of Kac type (\cite{Bra11}). \end{abstract}

\maketitle

\bibliographystyle{amsplain}

\section*{Introduction} A (classical) discrete group $\Gamma$ has the Haagerup property if (and only if) there is a net $(\varphi_i)$ of normalized positive definite functions in $C_0(\Gamma)$ converging pointwise to the constant function $1$. There are lots of examples of discrete groups with the Haagerup property: all amenable groups have this property. The free groups $F_N$ are examples of discrete groups with Haagerup property (see \cite{Haa79}) but which are not amenable. Thus, one says that the Haagerup property is a weak form of amenability. This property is also known as a ``strong negation" of Kazhdan's property $(T)$: the only (classical) discrete groups with both properties are finite. Another weak form of amenability is the weak amenability, see below for examples in the quantum setting. One can find more examples and a more complete approach to the problems and questions related to the Haagerup property, also called ``a-$T$-amenability", in \cite{CCJ}.

The Haagerup property has many interests in various fields of mathematics such as geometry of groups or functional analysis. We can mention e.g. groups with wall space structures (see \cite{CSV12} and \cite{CDCH10}) as illustrations of the interest in the Haagerup property with respect to the theory of geometry of groups. In functional analysis, the Haagerup property appears e.g. in questions related to the Baum-Connes conjecture (see \cite{HK01}) or in Popa's deformation/rigidity techniques (see \cite{Pop06}).

In \cite{Bra11}, a natural definition of the Haagerup property for compact quantum groups $\G$ of Kac type is proposed: $\G$ has the Haagerup approximation property if and only if its associated (finite) von Neumann algebra $L^{\infty}(\G)$ has the Haagerup property (see also below, Definition \ref{HaagCQG}). We use this definition with the slight modification: the \textit{dual} $\widehat{\G}$ of $\G$ has the Haagerup property if $L^{\infty}(\G)$ has the Haagerup property, so that this definition is closer to the classical case where $\widehat{\G}$ is a classical discrete group.
The author of \cite{Fre12} proposes another definition for the Haagerup property of discrete quantum groups: $\widehat{\G}$ has the Haagerup property if there exists a net $(a_i)$ in $c_0(\widehat{\G})$ which converges to $1$ pointwise and such that the associated multipliers $m_{a_i}$ are unital and completely positive. These approaches are equivalent in the unimodular case.

In \cite{Bra11} and \cite{Bra12}, the author shows that the duals of the compact quantum groups $O_N^+$, $U_N^+$ and $S_N^+$, introduced by Wang (see \cite{Wang2} and \cite{Wang}) have the Haagerup property. In fact, in \cite{Bra12}, it is proved that any trace-preserving quantum automorphism group of a finite dimensional $C^*$-algebra has the Haagerup property. In \cite{Fre12}, using some block decompositions and Brannan's proof of the fact that $\widehat{O_N^+}$ has the Haagerup property (precisely that some completely positive multipliers can be found), the author proves that $\widehat{O_N^+}$ is weakly amenable (in fact, it is also proved in \cite{Fre12} that the $\widehat{U_N^+}\subset\Z*\widehat{O_N^+}$ is weakly amenable, and an argument of monoidal equivalence allows to prove, in particular, that $\widehat{S_N^+}$ is weakly amenable too). In \cite{Fim10}, a definition of property $(T)$ for discrete quantum groups and some classical properties for discrete groups are generalized, for instance: discrete quantum groups with property $(T)$ are finitely generated and unimodular.

The aim of this paper is to prove that the duals of quantum reflection groups $H_N^{s+}$, introduced in \cite{BBCC11}, have the Haagerup property. It is a natural generalization of the case $s=1$ treated in \cite{Bra12} (since $H_N^{1+}=S_N^+$). However, this generalization is not immediate: as a matter of fact, the sub $C^*$-algebra generated by the characters is not commutative so that the strategy used in \cite{Bra11} and \cite{Bra12} does not work anymore. However, a fundamental tool of the proof of the main result of our paper is \cite[Theorem 3.7]{Bra11}.

What motivates our paper is also the fact that quantum reflection groups are free wreath products between $\Z_s$ and $S_N^+$ (see \cite{Bic04} and Theorem \ref{N23} below) and the result proved in our paper naturally leads to the following question: is it true that if $\Gamma$ is a discrete group which has the Haagerup property then $\widehat{\Gamma}\wr_wS_N^+$ has the Haagerup property ? One can notice the similarity with the result in \cite{CSV12} concerning (classical) wreath products of discrete groups: If $\Gamma, \Gamma'$ are countable discrete groups with the Haagerup property then $\Gamma\wr\Gamma'$ also has the Haagerup property. This similarity is however formal: in our paper we are considering (free) wreath products of groups whose \textit{duals} have the Haagerup property.

Our proof of the fact $\widehat{H_N^{s+}}$ have the Haagerup property relies on the knowledge of the fusion rules of the associated compact quantum group $H_N^{s+}$, determined in \cite{BV09}. Indeed, there is no general result about fusion rules for free wreath products of compact quantum groups yet.

The rest of the paper is organized as follows. In the section \ref{prelim}, we recall the definition of the Haagerup property for compact quantum groups of Kac type and we give the result of Brannan concerning the construction of normal, unital, completely positive and trace-preserving maps on $L^{\infty}(\G)$ (see Theorem \ref{fond}). We also give a positive answer to a question asked in \cite{Wor87}, in the discrete and Kac setting case, concerning symmetric tensors with respect to the coproduct. Then we collect some results on Tchebyshev polynomials: some are already mentioned and used in \cite{Bra11}, but we give suitably adapted statements and proofs for our purpose. Thereafter, we recall the definition of quantum reflection groups $H_N^{s+}$, and we describe their irreducible corepresentations and the fusion rules binding them. We also recall that at $s=1$, we get the quantum permutation groups $S_N^+$. In section \ref{corep}, we identify the images of the irreducible characters of $C(H_N^{s+})$ by the canonical morphism onto $C(S_N^+)$. In the section \ref{HAPHNS}, we prove that the duals of the quantum reflection groups $H_N^{s+}$ have the Haagerup approximation property for all $N\ge4$.

\section{Preliminaries}\label{prelim}

Let us first fix some notations. One can refer to \cite{Bra11}, \cite{VV07}, \cite{KV00} and \cite{Wor95} for more details. 
In this paper, $\G=(C(\G),\Delta)$ will denote a compact quantum group, where $C(\G)$ is a full Woronowicz $C^*$-algebra. Furthermore, every compact quantum group $\G$ considered in this paper is of Kac type (or equivalentely, its dual $\widehat{\G}$ is unimodular) that is: the unique Haar state $h$ on $C(\G)$, is tracial. (We recall that $L^{\infty}(\G)$ is defined by $L^{\infty}(\G)=C_r(\G)''=\pi_h(C(\G))''$, where $(L^2(\G),\pi_h)$ is the GNS construction associated to $h$).

\subsection{Haagerup property for compact quantum groups of Kac type}

\begin{defi}\label{HaagCQG}
The dual $\widehat{\G}$ of a compact quantum group $\G=(C(\G),\Delta)$ of Kac type has the Haagerup approximation property if the finite von Neumann algebra $(L^{\infty}(\G),h)$ has the Haagerup approximation property i.e. if there exists a net $(\phi_x)$ of trace preserving, normal, unital and completely positive maps on $L^{\infty}(\G)$ such that their unique extensions to $L^2(\G)$ are compact operators and $(\phi_x)$ converges to $id_{L^{\infty}(\G)}$ pointwise in $L^2$-norm.
\end{defi}

One essential tool to construct nets of normal, unital, completely positive and trace preserving maps (we will say \textit{NUCP trace preserving} maps) is the next theorem proved in \cite{Bra11}. We will denote by $Irr(\G)$ the set indexing the equivalence classes of irreducible corepresentations of a compact quantum group $\G$ and by $Pol(\G)$ the linear space spanned by the matrix coefficients of such corepresentations $u^{\alpha}, \alpha\in Irr(\G)$. If $\alpha\in Irr(\G)$, let $L_{\alpha}^2(\G)\subset L^2(\G)$ be the subspace spanned by the GNS images of matrix coefficients $u_{ij}^{\alpha},\ i,j\in\{1,\dots,d_{\alpha}\}$ of the irreducible unitary corepresentation $u^{\alpha}$ ($d_{\alpha}=\dim (u_{ij}^{\alpha}))$ and $p_{\alpha}: L^2(\G)\to L_{\alpha}^2(\G)$ be the associated orthogonal projection. Then $L^2(\G)=l^2-\bigoplus_{\alpha\in Irr(\G)}L_{\alpha}^2(\G)$. We denote by $C(\G)_0\subset C(\G)$ the $C^*$-algebra generated by the irreducible characters $\chi_{\alpha}=\sum_{i=1}^{d_{\alpha}}u_{ii}^{\alpha}$ of a compact quantum group $\G$ and $\chi_{\overline{\alpha}}$ the character of the associated conjugate corepresentation $u^{\overline{\alpha}}$.

\begin{thm}\label{fond}\cite[Theorem 3.7]{Bra11}
Let $\G=(C(\G),\Delta)$ be a compact quantum group of Kac type. Then for any state $\psi\in C(\G)_0^*$,
the map $$T_{\psi}=\sum_{\alpha\in Irr(\G)}\frac{\psi(\chi_{\overline{\alpha}})}{d_{\alpha}}p_{\alpha}$$ is a unital contraction on $L^2(\G)$ and the restriction of $T_{\psi}$ to $L^{\infty}(\G)$ defines a NUCP $h$-preserving map still denoted $T_{\psi}$.
\end{thm}

The averaging methods used to prove this theorem allow us to answer, in a restricted setting, a question asked in \cite{Wor87}.

Let $\G=(C(\G),\Delta)$ be a compact quantum group. Then consider the $C^*$-subalgebra $C(\G)_{\text{central}}:=\{a\in C(\G): \Delta(a)=\Sigma\circ\Delta(a)\}$ i.e. the $C^*$-subalgebra of the symmetric tensors in $C(\G)\otimes C(\G)$ with respect to $\Delta$ ($\Sigma$ denotes the usual flip map $\Sigma: C(\G)\otimes C(\G)\to C(\G)\otimes C(\G), a\otimes b\mapsto b\otimes a$).
In \cite{Wor87}, the author also defines $Pol(\G)_{\text{central}}:=\{a\in Pol(\G): \Delta(a)=\Sigma\circ\Delta(a)\}.$ We recall the question asked by Woronowicz (see \cite{Wor87} thereafter Proposition 5.11): 
\begin{quest}
Is $Pol(\G)_{\emph{central}}$ dense in $C(\G)_{\emph{central}}$ (for the norm of $C(\G)$) ? 
\end{quest}
Then the answer is yes, at least in the Kac and discrete setting. We simply denote by $||.||$ the norm on $C(\G)$.
It is clear, and proved in \cite{Wor87}, that $Pol(\G)_{\text{central}}=\text{span}\{\chi_{\alpha}: \alpha\in Irr(\G)\}$ where $\chi_{\alpha}=\sum_{i=1}^{d_{\alpha}}u_{ii}^{\alpha}$ denotes the character of an irreducible finite dimentional corepresentation $(u_{ij}^{\alpha})$. So the problem reduces to prove that $C(\G)_{\text{central}}\subset\overline{\text{span}}^{||.||}\{\chi_{\alpha}: \alpha\in Irr(\G)\},$
the other inclusion being clear.

\begin{thm}\label{centralWor}
Let $\G_r=(C(\G_r),\Delta_r)$ be a compact quantum group of Kac type with faithful Haar state. Then $\overline{Pol(\G_r)_{\emph{central}}}^{||.||}=C(\G_r)_{\emph{central}}.$
\end{thm}
\begin{proof}
We first note that $\Delta_r$ preserves the trace in the sense that $(h\otimes h)\circ
\Delta_r = h$. As a result the Hilbertian adjoint $\Delta_r^*$, of the $L^2$-extension of $\Delta_r$, is well-defined and we
have $||\Delta_r^*(x)||\le ||x||$ for $x\in C(\G_r)\otimes C(\G_r)$ with respect to the
operator norms (note that this is particular to the tracial
situation). Since $\Delta_r^*$ clearly maps the subspace $Pol(\G_r)\otimes Pol(\G_r)$ of
$L^2(\G_r)\otimes L^2(\G_r)$ to $Pol(\G_r)$, it also restricts to a contractive map
from $C(\G_r)\otimes C(\G_r)$ to $C(\G_r)$, still denoted $\Delta_r^*$. Now we put $E = \Delta_r^*\circ \Sigma\circ \Delta_r
: C(\G_r) \to C(\G_r)$. We have $||E||\le 1$, and for $a \in C(\G_r)_{\text{central}}, E(a) =
\Delta_r^*\circ \Delta_r(a) = a$ so that $C(\G_r)_{\text{central}} \subset E(C(\G_r))$.

But, on the other hand, for any matrix coefficient of a finite dimensional unitary corepresentations $(u_{ij}^{\alpha})$, we have
\begin{align*}
E(u_{ij}^{\alpha})&=\Delta_r^*\circ\Sigma\circ\Delta_r(u_{ij}^{\alpha})=\Delta_r^*\circ\Sigma\left(\sum_{k}u_{ik}^{\alpha}\otimes u_{kj}^{\alpha}\right)=\Delta_r^*\left(\sum_{k}u_{kj}^{\alpha}\otimes u_{ik}^{\alpha}\right).\\
\end{align*}
We compute $\Delta_r^*\left(\sum_ku_{kj}^{\alpha}\otimes u_{ik}^{\alpha}\right)$ using the duality pairing induced by the inner product coming from the Haar state $h$: let $\beta\in Irr(\G_r)$, then for all $1\le p,q\le d_{\beta}$:
\begin{align*}
\left<u_{pq}^{\beta},\Delta_r^*\left(\sum_ku_{kj}^{\alpha}\otimes u_{ik}^{\alpha}\right)\right>_{h}&=\sum_{l,k}\left<u_{pl}^{\beta}\otimes u_{lq}^{\beta},u_{kj}^{\alpha}\otimes u_{ik}^{\alpha}\right>_{h}=\frac{\delta_{\alpha\beta}\delta_{ij}\delta_{pq}}{d_{\alpha}^2}\\
&=\left <u_{pq}^{\beta},\frac{\delta_{ij}}{d_{\alpha}}\chi_{\alpha}\right>_{h}.
\end{align*}
Then summarizing, we have $E(u_{ij}^{\alpha})=\frac{\delta_{ij}}{d_{\alpha}}\chi_{\alpha}\in\text Pol(\G_r)_{\text{central}}$, $||E||=1$ and $E|_{Pol(\G_r)_{\text{central}}}=id$. Thus we obtain a conditional expectation $E: C(\G_r)\to \overline{Pol(\G_r)_{\text{central}}}^{||.||}=\overline{\text{span}}^{||.||}\{\chi_{\alpha}: \alpha\in Irr(\G_r)\}$.
But we have $C(\G_r)_{\text{central}}\subset E(C(\G_r))$ and the result follows.
\end{proof}

\begin{nota}
We will denote by $Pol(\G)_0$ and $C(\G)_0$ the central $*$-algebras and $C^*$-algebras generated by the irreducible characters of a compact quantum group $\G$.
\end{nota}
\bigskip

\subsection{Tchebyshev polynomials.}\label{ao}

\begin{defi}
We define a family of polynomials $(A_t)_{t\in\N}$ as follows: $A_0=1, A_1=X$ and for all $t\ge1$ 
\begin{equation}\label{recurform}
A_1A_t=A_{t+1}+A_{t-1}.
\end{equation}
We call them the dilated Tchebyshev polynomials of second kind.
\end{defi}

We will use the following results on Tchebyshev polynomials $A_t$. The second one is based upon a result proved in \cite[Proposition 4.4]{Bra11}, but suitably adapted to our purpose.

\begin{prp}\label{ChebRec}
for all $t,s\ge1$ we have:
$A_tA_s=A_{t+s}+A_{t-1}A_{s-1}$
\end{prp}

\begin{proof}
This result is easily proved by induction on $t\ge1$.
\end{proof}

\begin{prp}\label{MajCheb}
Let $N\ge2$. For all $x\in(2,N)$, there exists a constant $c\in(0,1)$ 
such that for all integers $t\ge1$ we have $$0<\frac{A_t(x)}{A_t(N)}\le \left(\frac{x}{N}\right)^{c t}.$$
\end{prp}

\begin{proof}
First, we follow the proof of \cite[Proposition 4.4]{Bra11} and introduce the function $q(x)=\frac{x+\sqrt{x^2-4}}{2}$, for $x>2$. Then an induction and the recursion formula (\ref{recurform}) for the polynomials $A_t$ show that for all $t\ge0$, we have 
\begin{equation*}
A_t(x)=\frac{q(x)^{t+1}-q(x)^{-t-1}}{q(x)-q(x)^{-1}}
\end{equation*}
Then using the same tricks as in \cite{Bra11}, we get that for all fixed $x\in(2,N)$ and all $t\ge1$
\smallskip
\begin{align*}
\frac{A_t(x)}{A_t(N)}&=\frac{q(x)^{t+1}-q(x)^{-t-1}}{q(N)^{t+1}-q(N)^{-t-1}}\frac{q(N)-q(N)^{-1}}{q(x)-q(x)^{-1}}\\
&=\left(\frac{x}{N}\right)^t\left(\frac{1+\sqrt{1-\frac{4}{x^2}}}{1+\sqrt{1-\frac{4}{N^2}}}\right)^t\frac{1-q(x)^{-2t-2}}{1-q(N)^{-2N-2}}\frac{1-q(N)^{-2}}{1-q(x)^{-2}}.\\
\end{align*}
Now notice that the factor $\frac{1-q(x)^{-2t-2}}{1-q(N)^{-2N-2}}$ is less than $1$ because $q$ is increasing. Furthermore, we have $$\left(\frac{1+\sqrt{1-\frac{4}{x^2}}}{1+\sqrt{1-\frac{4}{N^2}}}\right)^t\frac{1-q(N)^{-2}}{1-q(x)^{-2}}\underset{t\to\infty}\longrightarrow0$$ since the last factor does not depend on $t$ and $\frac{1+\sqrt{1-\frac{4}{x^2}}}{1+\sqrt{1-\frac{4}{N^2}}}<1$. Hence, there exists $t_0$ such that 
$
\dfrac{A_t(x)}{A_t(N)}\le\left(\dfrac{x}{N}\right)^t$
for all $t\ge t_0$. It remains to show that there exists $c\in(0,1)$ such that $\dfrac{A_t(x)}{A_t(N)}\le\left(\frac{x}{N}\right)^{c t_0}$ for all $t=1,\dots, t_0-1$, since for all $0<t<t_0$, $\left(\dfrac{x}{N}\right)^{ct_0}\le\left(\dfrac{x}{N}\right)^{ct}$ . To prove that such a $c$ exists, we notice that 
$
\max\left\{\frac{A_t(x)}{A_t(N)}: t=1,\dots,t_0-1\right\}
:=D<1
$
since the Tchebyshev polynomials are increasing on $(2,+\infty)$. Hence, it is clear that we can find $c>0$ such that $\left(\dfrac{x}{N}\right)^{ct_0}\ge D$.
\end{proof}
\begin{rque}
\begin{enumerate}
\item In \cite[Proposition 6.4]{Bra11}, the exponent is better (there is no constant $c$) but there is a constant multiplying $\left(\dfrac{x}{N}\right)^t$. Our version allows an easy proof of Proposition \ref{c05} below.
\item The previous proposition gives information on the behavior of the dilated Tchebyshev polynomials on $(2,+\infty)$: the quotient $\dfrac{A_t(x)}{A_t(N)}$ has an exponential decay with respect to $t\ge1$. We will also need some informations on this quotient when $x\in(0,2)$ and $N=2$. That is the aim of the next paragraph.
\end{enumerate}

\end{rque}

The polynomials $A_t$ are linked to the Tchebyshev polynomials of second kind $U_t$ by the following formula: $\forall t\in\N,x\in[0,1], A_t(2x)=U_t(x).$ Indeed, we recall (see \cite{Riv90} for more details) that the Tchebyshev polynomials of second kind $U_t$ are defined for all $x\in[-1,1]$ by 
\begin{equation}\label{tchebu}
U_t(x)=\frac{\sin((t+1)\arccos(x))}{\sqrt{1-x^2}}=\frac{\sin((t+1)\theta)}{\sin(\theta)},\ \ \text{ with } x=\cos(\theta).
\end{equation}

In particular, $U_0=1$, $U_1(x)=2x$ and for all $t\in\N^*, U_t(1)=t+1$. Then one can check that for all $t\in\N$ and $x\in[0,1]$: $2xU_{t}(x)=U_{t+1}(x)+U_{t-1}(x)$. 

\begin{prp}\label{propN=4}
Let $x\in(0,2)$. Then for any integer $t\ge1$
\begin{equation*}\label{N=4good}
\left|\frac{A_t(x)}{A_t(2)}\right|=\frac{1}{t+1}\left|\frac{\sin((t+1)\theta)}{\sin(\theta)}\right|,\ \ \text{ with } x=\cos(\theta).
\end{equation*}
In particular, there exists a positive constant $D<1$ such that $\forall t\ge1, \left|\dfrac{A_t(x)}{A_t(2)}\right|\le D$
and
$\dfrac{A_t(x)}{A_t(2)}\longrightarrow 0 \text{ as } t\to\infty.$
\end{prp}

\begin{proof}
First, by what we recalled above, we can write $\dfrac{A_t(x)}{A_t(2)}=\dfrac{U_t(\frac{x}{2})}{U_t(1)}=\dfrac{U_t(\frac{x}{2})}{t+1}.$ Thus, if $x=2\cos(\theta)$, we have by the relation (\ref{tchebu}) above 
\begin{equation*}\label{quot}
\left|\frac{A_t(x)}{A_t(2)}\right|=\frac{1}{t+1}\left|\frac{\sin((t+1)\theta)}{\sin(\theta)}\right|\underset{t\to\infty}{\longrightarrow}0.
\end{equation*}
On the other hand, on $[0,1]$, the polynomials $U_t, t\ge1$ have $t+1$ as a maximum, only attained in $1$. 
Then, it is clear that for all $t\ge1$ and $x\in(0,2)$: 
$$0<\frac{A_t(x)}{A_t(2)}=\frac{U_t(\frac{x}{2})}{t+1}<1.$$
So the existence of the announced constant $D$ is clear.
\end{proof}

\bigskip

\subsection{Quantum reflection groups}\label{Ash}

In this subsection, we recall the definition of the quantum reflection groups $H_N^{s+}$ and the particular case of the quantum permutation groups $S_N^+$. We also recall that $C(H_N^{s+})$ is the free wreath product of two quantum permutation algebras. In the end of this subsection, we recall the description of the irreducible corepresentations of $C(H_N^{s+})$ together with the fusion rules binding them.

\begin{defi}\label{ahs}\cite[Definition 11.3]{BBCC11}
Let $N\ge2, s\ge 1$ be integers. The quantum reflection group $H_N^{s+}$ is the pair $(C(H_N^{s+}),\Delta)$ composed of the universal $C^*$-algebra generated by $N^2$ normal elements $U_{ij}$ satisfying the following relations 
\begin{enumerate}
\item $U=(U_{ij})$ is unitary,
\item $^tU=(U_{ji})$ is unitary,
\item $p_{ij}=U_{ij}U_{ij}^*$ is a projection,
\item $U_{ij}^s=p_{ij}$
\end{enumerate}
together with the coproduct $\Delta: C(H_N^{s+})\to C(H_N^{s+})\otimes C(H_N^{s+})$ given by 
$$\Delta(U_{ij})=\sum_kU_{ik}\otimes U_{kj}.$$
\end{defi}

\begin{rque}\label{as}\
\begin{enumerate}
\item For $s=1$ we get the quantum permutation group $S_N^+$. The definition of $S_N^+$ thus may be summed up as follows (see also \cite{Wang}): $S_N^+$ is the pair $(C(S_N^+),\Delta)$ where  
\begin{enumerate}
\item $C(S_N^+)$ is the universal $C^*$-algebra generated by $N^2$ elements $v_{ij}$ such that the matrix $v=(v_{ij})$ is unitary and $v_{ij}=v_{ij}^*=v_{ij}^2$ (i.e. $v$ is a magic unitary).
\item The coproduct is given by the usual relations making of $v$ a corepresentation (the fundamental one) of $C(S_N^+)$.
\end{enumerate}
\item For $s=2$, we find the hyperoctahedral quantum group, i.e. the easy quantum group $H_N^+$ studied e.g. in \cite{Web12}.
\item\label{morphlater} There is a morphism $C(H_N^{s+})\to C(S_N^+)$ of compact quantum groups: one only has to check that the generators $v_{ij}$ of $C(S_N^+)$, satisfy the relations described in Definition \ref{ahs}, which is clear. 
\end{enumerate}
\end{rque}
\begin{nota}\label{pi}
We will denote by $\pi: C(H_N^{s+})\to C(S_N^+)$ the canonical arrow mentioned in the remark above.
\end{nota}

Here are the results concerning the irreducible corepresentations of $C(S_N^+)$:

\begin{thm}\label{BanSN}\cite[Theorem 4.1]{Ban99}
There is a maximal family $\left(v^{(t)}\right)_{t\in\N}$ of pairwise inequivalent irreducible finite dimensional unitary representations of $S_N^+$ such that:
\begin{enumerate}
\item $v^{(0)}=1$ and $v$ is equivalent to $1\oplus v^{(1)}$.
\item The conjugate of any $v^{(t)}$ is equivalent to itself that is $\overline{v^{(t)}}\simeq v^{(t)},\ \forall t\in\N$.
\item The fusion rules are the same as for $SO(3)$: $$v^{(s)}\otimes v^{(t)}\simeq\bigoplus_{k=0}^{2\min(s,t)}v^{(s+t-k)}$$
\end{enumerate}
\end{thm}
We denote by $\chi_k=\sum_{i=1}^{d_k}v_{ii}^{(k)}$ the character associated to $v^{(k)}$.

We will need the following proposition, proved in \cite{Bra12}:
\begin{prp}\label{branPi}
Let $\chi$ be the character associated to the fundamental corepresentation $v$ of $C(S_N^+)$. Then, $\chi^*=\chi$ and there is a $*$-isomorphism $C^*(\chi)=C(S_N^+)_0=C^*(\chi_t: t\in\N)\simeq C([0,N])$ identifying $\chi_t$ to the polynomial defined $\Pi_t$ by $\Pi_0=1, \Pi_1=X-1$ and $\forall t\ge1, \Pi_1\Pi_t=\Pi_{t+1}+\Pi_t+\Pi_{t-1}$.
\end{prp}

\begin{rque}
\begin{enumerate}
\item The recursion formula defining the polynomials $\Pi_t$ is the one satisfied by the irreducible characters $\chi_t$.
\item The polynomials $A_t$ and $\Pi_t$ are linked by the formula: $\Pi_t(x)=A_{2t}(\sqrt{x})$.
\end{enumerate}
\end{rque}

Before describing the fusion rules of $C(H_N^{s+})$, we recall that these compact quantum groups are free wreath products:
\begin{thm}\label{N23}\cite[Theorem 3.4]{BV09} Let $N\ge2$, then we have the following isomorphisms of compact quantum groups:
\begin{enumerate}
\item[-] $C(H_N^{s+})\simeq C(\Z_s)*_wC(S_N^+)=C^*(\Z_s^{*N})*C(S_N^+)/<[z_i,v_{ij}]=0>$ where $z_i$ is the generator of the $i$-th copy $\Z_s$ in the free product $\Z_s^{*N}$.
\item[-] In particular $C(H_2^{s+})\simeq C(\Z_s)*_wC(Z_2)$, $C(H_3^{s+})\simeq C(\Z_s)*_wC(S_3)$.
\end{enumerate}
\end{thm}
Let us now give the description of the irreducible corepresentations of $C(H_N^{s+})$. 
\begin{thm}\label{basic} \cite[Theorem 4.3, Corollary 6.4]{BV09}
$C(H_N^{s+})$ has a unique family of $N$-dimensional corepresentations (called basic corepresentations) $\{U_k: k\in\Z\}$, satisfying the following conditions:
\begin{enumerate}
\begin{minipage}[t]{0.4\linewidth}
\item $U_k=(U_{ij}^k)$ for any $k>0$.
\item $U_k=U_{k+s}$ for any $k\in\Z$.
\item $\overline{U}_k=U_{-k}$ for any $k\in\Z$.
\end{minipage}
\begin{minipage}[t]{0.4\linewidth}
\item $U_1,\dots,U_{s-1}$ are irreducible.
\item $U_0=1\oplus\rho_0$, $\rho_0$ irreducible.
\item $\rho_0,U_1,\dots,U_{s-1}$ are inequivalent corepresentations.
\end{minipage}
\end{enumerate} 
\end{thm}

\begin{nota}\label{basiks}
We will denote the basic irreducible corepresentations of $C(H_N^{s+})$ by $\rho_t, t\in\{0,\dots,s-1\}$, with $\rho_t=U_t\ \forall t\in\{1,\dots,s-1\}$ and $\rho_0=U_0\ominus 1$ (where $U_0=(U_{ij}U_{ij}^*)$).
\end{nota}

The proof of the first three assertions follows from the definitions of corepresentations of compact quantum groups and of the definition of $C(H_N^{s+})$. The proof of the last three assertions is based upon Woronowicz's Tannaka-Krein duality (see \cite{Wor88}) and methods inspired by \cite{Ban96}, \cite{Ban99} and \cite{BBCC11}. Now, we can give the description of the fusion rules:

\begin{thm}\label{corep2}\cite[Theorem 8.2]{BV09}
Let $M$ be the monoid $M=\langle a,z: z^s=1\rangle$ with involution $a^*=a,\ z^*=z^{-1}$, and the fusion rules obtained by recursion from the formulae 
\begin{equation}\label{monoidrelation}
vaz^i\otimes z^{j}aw=vaz^{i+j}aw\oplus\delta_{i+j,0}\left(v\otimes w\right)
\end{equation}

Then the irreducible corepresentations $r_{\alpha}$ of $C(H_N^{s+})$ can be indexed by the elements $\alpha$ of the submonoid $S$ generated by the elements $az^{i}a, i=0,\dots,s-1$, with involution and fusion rules above. 
\end{thm}

\begin{rque}\label{notacorep}\
\begin{enumerate}
\item $S$ is composed of elements $a^{L_1}z^{J_1}\dots z^{J_{K-1}}a^{L_K}$ with 
\begin{enumerate}
\item[-] $J_i,L_i>0$ integers.
\item[-] $L_1$, $L_K$ odd integers and all the $L_i$'s$,\ i\in\{2,\dots,K-1\}$ even integers.
\item[-] Except if $K=1$, then $L_K$ is an even integer.
\end{enumerate}
\item With this description, we can identify the basic corepresentations introduced above: the corepresentation $r_{a^2}$ is the corepresentation $\rho_{{0}}=(U_{ij}U_{ij}^*)\ominus1$ and for $t\ne0$, $r_{az^ta}$ is the corepresentation $\rho_{{t}}=(U_{ij}^t)$.
\item In Proposition \ref{Image}, we will use the suggestive notation $$vaz^{i+j}aw=(vaz^i\otimes z^{j}aw)\ominus\delta_{i+j,0}(v\otimes w),$$
which simply means that we have the relation (\ref{monoidrelation}) in the monoid $S$. 
\item If $\alpha=a^{L_1}z^{J_1}\dots z^{J_{K-1}}a^{L_K}\in S$, then the conjugate corepresentation of $r_{\alpha}$ is indexed by $\overline{\alpha}=a^{L_K}z^{-J_{K-1}}\dots z^{-J_{1}}a^{L_1}$
\end{enumerate}
\end{rque}
We end this subsection by the following proposition which summarizes the results above:
\begin{prp}\label{imcorep}
The canonical morphism $\pi: C(H_N^{s+})\to C(S_N^+)$ maps all the corepresentations $U_t,t\in\Z$ onto the fundamental corepresentation $v$ of $C(S_N^+)$ ; in other words, it maps all $\rho_t=r_{az^ta}, t\ne0$ onto $v$ and $\rho_0=r_{a^2}$ onto $v^{(1)}$.
\end{prp}
\bigskip

\section{Characters of quantum reflection groups and quantum permutation groups}\label{corep}

As announced in the introduction, we find the images of the irreducible characters of $C(H_N^{S+})$ under the canonical morphism $\pi: C(H_N^{s+})\to C(S_N^+)$.

\begin{prp}\label{Image}
Let $\chi_\alpha$ be the character of an irreducible corepresentation $r_{\alpha}$ of $C(H_N^{s+})$. Write $\alpha=a^{l_1}z^{j_1}\dots z^{j_{k-1}}a^{l_k}.$ Then, identifying $C(S_N^+)_0$ with $C([0,N])$, the image of $\chi_{\alpha}$, say $P_{\alpha}$, satisfies:
$$P_{\alpha}(X^2)=\pi(\chi_{\alpha})(X^2)=\prod_{i=1}^kA_{l_i}(X).$$
\end{prp}

\begin{proof}
We shall prove this proposition by induction on the even integer $\sum_{i=1}^kl_i$ using the description of the fusion rules given by Theorem \ref{corep2}, the recursion formula satisfied by the Tchebyshev polynomials, Proposition \ref{ChebRec} and Proposition \ref{imcorep}.
 
Let HR($\lambda$) be the following statement: $\pi(\chi_{\alpha})(X^2)=\prod_{i=1}^kA_{l_i}(X)$ for any $\alpha=a^{l_1}z^{j_1}\dots z^{j_{k-1}}a^{l_k}$ such that $2\le \sum_{i}l_i\le \lambda$.

Let us begin by studying simple examples (and initializing the induction).

Consider the element $aza$. Then, the irreducible corepresentation $r_{aza}$ (written $\rho_{1}$ in Notation \ref{basiks}) is sent by $\pi$ onto $v=1\oplus v^{(1)}$ by Proposition \ref{imcorep}. Thus, in term of characters, we obtain by Proposition \ref{branPi} $$\pi(\chi_{aza})(X)=1+(X-1)=X=A_1(X)$$ i.e. $$P_{aza}(X^2)=X^2=A_1(X)A_1(X).$$ Actually, this holds for all elements $\alpha=az^ja,\ j\in\{1,\dots,s-1\}$ (since every irreducible corepresentation $r_{az^ja}$ is sent by $\pi$ onto $1\oplus v^{(1)}$, as is $r_{aza}$).  

Consider the element $a^2$. Then, the irreducible corepresentation $r_{a^2}$ (written $\rho_0$ in Notation \ref{basiks}) is sent by $\pi$ onto $v^{(1)}$. Thus $\pi(\chi_{a^2})(X)=X-1$. i.e. $$P_{a^2}(X^2)=X^2-1=A_2(X).$$

To prove HR($2$) one has to show that 
$\pi(\chi_{a^2})(X^2)=A_2(X)$ and $\pi(\chi_{az^{j}a})(X^2)=A_1A_1(X)$
for all $j\in\{1,\dots,s-1\}$, what we have just done above.

Now assume HR($\lambda$) holds: $\pi(\chi_{\beta})(X^2)=\prod_{i=1}^kA_{l_i}(X)$ for any $\beta=a^{l_1}z^{j_1}\dots z^{j_{k-1}}a^{l_k}$ such that $2\le \sum_{i}l_i\le \lambda$. We now show HR($\lambda+2$).

Let $\alpha=a^{L_1}z^{J_1}\dots a^{L_K}$, with $\sum_iL_i=\lambda+2$. In order to use HR($\lambda$), we must ``break'' $\alpha$ using the fusion rules as in the examples above. 
Then, essentially, one has to distinguish the cases $L_K=1, L_K=3$ and $L_K\ge5$ (in the case $L_K\ge5$ we can ``break $\alpha$ at $a^{L_{K}}$" but in the other cases we must use $a^{L_{K-1}}$ or $a^{L_{K-2}}$ if they exist, that is if there are enough factors $a^{L_i}$). So first, we deal with two special cases below, in order to have ``enough'' factors $a^L$ in $\alpha$ in the sequel. We use the fusion rules described in Theorem \ref{corep2} (and the notations described after, see Remark \ref{notacorep}).

\begin{enumerate}

\item[-] If $K=1$ i.e. $L_K=\lambda+2,\ J_i=0\ \forall i$, write: $$\alpha=a^{\lambda+2}=(a^{\lambda}\otimes a^2)\ominus (a^{\lambda-1}\otimes a)=(a^{\lambda}\otimes a^2)\ominus a^{\lambda}\ominus a^{\lambda-2}.$$ Then using the hypothesis of induction and Proposition \ref{ChebRec}, we get 

\begin{align*}
\pi(\chi_{\alpha})(X^2)&=A_{\lambda}A_2(X)-A_{\lambda}(X)-A_{\lambda-2}(X)\\
&=A_{\lambda}A_2(X)-(A_{\lambda}(X)+A_{\lambda-2}(X))\\
&=A_{\lambda}A_2(X)-A_{\lambda-1}A_1(X)\\
&=A_{\lambda+2}(X).
\end{align*}
(Notice that if $\lambda=2$ one has $\lambda-2=0$ and $a^4=(a^2\otimes a^2)\ominus (a\otimes a)=(a^2\otimes a^2)\ominus a^2\ominus 1$ so that the result we want to prove then is still true.)

\item[-] If $K=2, J:=J_1\ne0$, write $\alpha=a^{L_1}z^{J}a^{L_2}$. We have $L_1+L_2= \lambda+2\ge4$ and $L_1,L_2$ are odd hence $L_1$ or $L_2\ge3$, say $L_1\ge3$. Write $$a^{L_1}z^{J}a^{L_2}=(a^2\otimes a^{L_1-2}z^Ja^{L_2})\ominus (a\otimes a^{L_1-3}z^Ja^{L_2}).$$

If $L_1=3$ then the tensor product $a\otimes a^{L_1-3}z^Ja^{L_2}$ is equal to $az^Ja^{L_2}$ hence $\alpha=a^3z^Ja^{L_2}$ satisfies 
\begin{align*}
\pi(\chi_{\alpha})(X^2)&=A_2A_{1}A_{L_2}(X)-A_1A_{L_2}(X)\\
&=A_3(X)A_{L_2}(X).
\end{align*}

If $L_1>3$ (i.e. $L_1\ge5$), then the tensor product $a\otimes a^{L_1-3}z^Ja^{L_2}$ is equal to $a^{L_1-2}z^Ja^{L_2}\oplus a^{L_1-4}z^Ja^{L_2}$. We get
\begin{align*}
\pi(\chi_{\alpha})(X^2)&=A_2A_{L_1-2}A_{L_2}(X)-A_{L_1-2}A_{L_2}(X)-A_{L_1-4}A_{L_2}(X)\\
&=A_{L_1}(X)A_{L_2}(X).
\end{align*}

\item[-] From now on, we suppose that there are more than three factors $a^{L_i}$ in $\alpha$ i.e. $K\ge3$. We will have to distinguish three cases: $L_K=1, L_K=3$ and $L_K\ge5$.

If $5\le L_K<\sum_iL_i$, write $L_K=m_K+2$. Then we have $m_K\ge3$, so 

\begin{align*}
a^{L_1}z^{J_1}\dots a^{L_K}&=a^{L_1}z^{J_1}\dots a^{m_K+2}\\
&=(a^{L_1}z^{J_1}\dots a^{m_K}\otimes a^2)\ominus (a^{L_1}z^{J_1}\dots a^{m_K-1}\otimes a)\\
&=(a^{L_1}z^{J_1}\dots a^{m_K}\otimes a^2)\ominus a^{L_1}z^{J_1}\dots a^{m_K}\ominus a^{L_1}z^{J_1}\dots a^{m_K-2}.
\end{align*}
Then 
\begin{align*}
\pi(\chi_{\alpha})(X^2)&=A_{L_1}\dots A_{L_{K-1}}A_{m_k}A_2(X)-A_{L_1}\dots A_{m_K}(X)-A_{L_1}\dots A_{m_K-2}(X)\\
&=A_{L_1}\dots A_{L_{K-1}}A_{L_K}(X).\\
\end{align*}

If $m_K=1$, i.e. $L_K=3$, we proceed in the same way using

\begin{align*}
a^{L_1}z^{J_1}\dots z^{J_{K-1}}a^3=(a^{L_1}z^{J_1}\dots a\otimes a^2) \ominus a^{L_1}z^{J_1}\dots z^{J_{K-1}}a.
\end{align*}

To conclude the induction, one has to deal with the case $L_K=1$. We have to distinguish the following cases:

If $L_{K-1}\ge4$. We have

\begin{align*}
a^{L_1}z^{J_1}\dots a^{L_{K-1}}z^{J_{K-1}}a&=(a^{L_1}z^{J_1}\dots a^{L_{K-1}-1}\otimes az^{J_{K-1}}a) \ominus (a^{L_1}z^{J_1}\dots a^{L_{K-1}-2}\otimes z^{J_{K-1}}a)\\
&=(a^{L_1}z^{J_1}\dots a^{L_{K-1}-1}\otimes az^{J_{K-1}}a) \ominus a^{L_1}z^{J_1}\dots a^{L_{K-1}-2}z^{J_{K-1}}a.
\end{align*}
Then
\begin{align*}
\pi(\chi_{\alpha})(X^2)&=A_{L_1}\dots A_{L_{K-1}-1}A_{1}A_1(X)-A_{L_1}\dots A_{L_{K-1}-2}A_1(X)\\
&=A_{L_1}\dots A_{L_{K-1}}A_1(X).\\
\end{align*}

If $L_{K-1}=2$ and ${J_{K-1}}+{J_{K-2}}=0 \text{\ mod\ } s$, we can proceed in the same way using 

\begin{align*}
a^{L_1}z^{J_1}\dots &a^{L_{K-2}}z^{J_{K-2}}a^2z^{J_{K-1}}a\\
&=(a^{L_1}z^{J_1}\dots a^{L_{K-2}}z^{J_{K-2}}a\otimes az^{J_{K-1}}a)\ominus a^{L_1}z^{J_1}\dots a^{L_{K-2}+1}\ominus a^{L_1}z^{J_1}\dots a^{L_{K-2}-1}.\\
\end{align*}

The last case to deal with is $L_{K-1}=2$ and ${J_{K-1}}+{J_{K-2}}\ne0 \text{\ mod\ } s$, and again we can conclude thanks to

\begin{align*}
a^{L_1}z^{J_1}\dots z^{J_{K-2}}a^2z^{J_{K-1}}a&=(a^{L_1}\dots a^{L_{K-2}}z^{J_{K-2}}a\otimes az^{J_{K-1}}a) \ominus a^{L_1}z^{J_1}\dots a^{L_{K-2}}z^{J_{K-2}+J_{K-1}}a.\\
\end{align*}
\end{enumerate}
\end{proof}

As a corollary, we can get the result also proved in \cite{BV09} (see Theorem 9.3):

\begin{crl}\label{crlsurj}
Let $r_{\alpha}$ be an irreducible corepresentation of $C(H_N^{s+})$ with $\alpha=a^{l_1}z^{j_1}\dots a^{l_k}$. Then $$\emph{dim}(r_{\alpha})=\prod_{i=1}^kA_{l_i}(\sqrt{N}).$$
\end{crl}

\begin{proof}
We have $\text{dim}(r_{\alpha})=\epsilon_{C(H_N^{s+})}(\chi_{\alpha})=\epsilon_{C(S_N^+)}\circ\pi(\chi_{\alpha})$ since $\pi$ is a morphism of Hopf algebras. But the counit on $C(S_N^+)_0$ is given by the evaluation in $N$. Indeed, an immediate corollary of Theorem \ref{BanSN} and Proposition \ref{branPi}, is $\epsilon(\Pi_t)=\Pi_t(N)$ for all polynomials $\Pi_t$, which form a basis of $\R[X]$. Now by the previous proposition $\pi(\chi_{\alpha})(x)=\prod_{i=1}^kA_{l_i}(\sqrt{x})$, then $\epsilon_{C(S_N^+)}\circ\pi(\chi_{\alpha})=\prod_{i=1}^kA_{l_i}(\sqrt{N})$.
\end{proof}
\bigskip

\section{Haagerup property for quantum reflection groups}\label{HAPHNS}
In this section we show that duals of the quantum reflection groups $C(H_N^{s+})=C(\Z_s)*_wC(S_N^+),\ s\ge1$ have the Haagerup property for $N\ge4$.

We still denote by $\pi$ the canonical surjection $\pi: C(H_N^{s+})\to C(S_N^+)$ and by $\psi_x=ev_x$ the states on $C(S_N^+)_0\simeq C([0,N])$ used to show that $C(S_N^+)$ have the Haagerup property (see \cite{Bra12}). Essentially, we are going to use both morphisms $\pi, \psi_x$ in this way: we can define states $\phi_x$ composing  these maps, $\psi_x\circ\pi$, where $\pi$ sends characters of $C(H_N^{s+})$ on characters of $C(S_N^+)$. Thus, we obtain states on the central algebra $C(H_N^{s+})_0$ and, after checking that these states have some decreasing properties, we can use the Theorem \ref{fond} and conclude.

\begin{lem}
Let $\psi_x,\ x\in[0,N]$ be the states given by the evaluation in $x$ on the central $C^*$-algebra $C(S_N^+)_0$. Then for all $x\in[0,N]$, $\phi_x=\psi_x\circ\pi$ is a state on $C(H_N^{s+})_0$.
\end{lem}

\begin{proof}
One just has to note that $\pi$ is Hopf $*$-homomorphism and hence sends $C(H_N^{s+})_0$ to $C(S_N^+)_0$. Then $\psi_x\circ\pi$ is indeed a functional on $C(H_N^{s+})_0$. The rest is clear.
\end{proof}

\begin{nota}\label{notac0mon}
We introduce a proper function on the monoid $S$ (see Theorem \ref{corep2}). Let $L$ be defined by $L(\alpha)=\sum_{i=1}^{k_{\alpha}}l_i$ for $\alpha=a^{l_1}z^{j_1}\dots a^{l_{k_{\alpha}}}$. Notice that for all $R>0$ the set $B_R=\left\{\alpha=a^{l_1}z^{j_1}\dots a^{l_{k_{\alpha}}}: L(\alpha)=\sum_{i=1}^{k_{\alpha}}l_i\le R\right\}\subset S$ is finite. Thus we get that a net $(f_{\alpha})_{\alpha\in S}$ belongs to $c_0(S)\Longleftrightarrow\forall\epsilon>0\ \exists R>0: \forall\alpha\in S, (L(\alpha)>R\Rightarrow|f_{\alpha}|<\epsilon)$. We say that a net $(f_{\alpha})_{\alpha}$ converges to $0$ as $\alpha\to\infty$ if $(f_{\alpha})_{\alpha}\in c_0(S)$.
\end{nota}

\begin{prp}\label{c05}
Let $N\ge5$ and let $\chi_{\alpha}$ be an irreducible character of $C(H_N^{s+})$ associated to the irreducible corepresentation $r_{\alpha}$ with ${\alpha}=a^{l_1}z^{j_1}a^{l_2}\dots a^{l_{k_{\alpha}}}$. Then for all $x\in [0,N]$ 
$$C_{\alpha}(x):=\frac{\phi_x(\chi_{\alpha})}{\emph{dim}(r_{\alpha})}=\frac{\psi_x\circ\pi(\chi_{\alpha})(X)}{\emph{dim}(r_{\alpha})}=\prod_{i=1}^{k_{\alpha}}\frac{A_{l_i}(\sqrt{x})}{A_{l_i}(\sqrt{N})}.$$
Moreover $C_{\alpha}(x)$ converges to $0$ as $\alpha\to\infty$ for all $x\in[4,N)$.
\end{prp}

\begin{proof} Let $\alpha=a^{l_1}z^{j_1}\dots a^{l_{k_{\alpha}}}$. We obtain the first assertion using Proposition \ref{Image} and Corollary \ref{crlsurj}:
$$\pi(\chi_{\alpha})(X)=A_{l_1}\dots A_{l_{k_{\alpha}}}(\sqrt{X}),$$
$$d_{\alpha}:=\text{dim}(r_{\alpha})=\prod_{i=1}^{k_{\alpha}}A_{l_i}(\sqrt{N}).$$

By Proposition \ref{MajCheb}, for any fixed $x\in(4,N)$, there exists a constant $0<c<1$ such that $\dfrac{A_l(\sqrt{x})}{A_l(\sqrt{N})}\le\left(\dfrac{\sqrt{x}}{\sqrt{N}}\right)^{c l}$
for all $l\ge1$. Then 
$$C_{\alpha}(x)=\frac{\phi_x(\chi_{\alpha})}{\text{dim}(r_{\alpha})}=\prod_{i=1}^{k_{\alpha}}\frac{A_{l_i}(\sqrt{x})}{A_{l_i}(\sqrt{N})}\le\left(\frac{x}{N}\right)^{\frac{c}{2}\sum_{i}l_i}=\left(\frac{x}{N}\right)^{\frac{c}{2}L(\alpha)}\underset{\alpha\to\infty}\longrightarrow0.$$
\end{proof}
\bigskip

\begin{prp}\label{c04}(Case $N=4$)
Let $\chi_{\alpha}$ be an irreducible character of $C(H_4^{s+})$ associated to the irreducible corepresentation $r_{\alpha}$ with ${\alpha}=a^{l_1}z^{j_1}a^{l_2}\dots a^{l_{k_{\alpha}}}$. Then for all $x\in [0,4]$ 
$$C_{\alpha}(x):=\frac{\phi_x(\chi_{\alpha})}{\emph{dim}(r_{\alpha})}=\frac{\psi_x\circ\pi(\chi_{\alpha})(X)}{\emph{dim}(r_{\alpha})}=\prod_{i=1}^{k_{\alpha}}\frac{A_{l_i}(\sqrt{x})}{A_{l_i}(2)}.$$
Moreover $C_{\alpha}(x)$ converges to $0$ as $\alpha\to\infty$ for all $x\in(0,4)$.
\end{prp}
\begin{proof}
The proof of the first assertion is similar to the one of the previous proposition. For the second assertion, we use Proposition \ref{propN=4}. We recall that we proved in that proposition that there exists a constant $D<1$ such that for all $x\in(0,4)$ and all $l\ge1$ 
\begin{equation}\label{D}
\frac{A_{l}(\sqrt{x})}{A_{l}(2)}\le D.
\end{equation}

Let $\epsilon>0$ and $x\in(0,4)$. We want to prove that, 
\begin{equation}\label{want}
\prod_{i=1}^{k_{\alpha}}\frac{A_{l_i}(\sqrt{x})}{A_{l_i}(2)}<\epsilon\ \text{  for $\alpha$ large enough. }
\end{equation}
By (\ref{D}), there exists a $K>0$ such that 
$
\prod_{i=1}^{k_{\alpha}}\frac{A_{l_i}(\sqrt{x})}{A_{l_i}(2)}<\epsilon
$
for all $\alpha\in S$ with $k_{\alpha}\ge K$. But by Proposition \ref{propN=4} there is also an $L>0$ such that $\dfrac{A_{l}(\sqrt{x})}{A_{l}(2)}<\epsilon$ for all $l\ge L$, since this quotient converges to $0$. 

Now let $\alpha=a^{l_1}z^{j_1}\dots a^{l_{k_{\alpha}}}\in S$, with $L(\alpha)\ge LK$. Then either $k_{\alpha}\ge K$, or there exists $i_0\in\{1,\dots,k_{\alpha}\}$ such that $l_{i_0}\ge L$. In both case we can get (\ref{want}) since $$\prod_{i=1}^{k_{\alpha}}\frac{A_{l_i}(\sqrt{x})}{A_{l_i}(2)}\le\frac{A_{l_{i_0}}(\sqrt{x})}{A_{l_{i_0}}(2)},$$ each factor of the product being less that one.
\end{proof}

Then we can prove the theorem:
\begin{thm}\label{N5}
The dual of $H_N^{s+}$ has the Haagerup property for all $N\ge4$.
\end{thm}
\begin{proof}
We follow the proof in \cite{Bra11} for $O_N^+$. We prove that the dual of $H_N^{s+}$ has the Haagerup approximation property for all $N\ge4$ using both previous propositions. 
Consider the net $\left(T_{\phi_x}\right)_{x\in I_N}$ with $I_N=(4,N) \text{ if } N\ge5,\ I_N=(0,4) \text{ if } N=4$ and $$T_{\phi_x}=\sum_{\alpha\in Irr(H_N^{s+})}\frac{\phi_x(\chi_{\overline{\alpha}})}{d_{\alpha}}p_{\alpha}$$
The $\phi_x$ are states on $C(H_N^{s+})_0$ so, by Theorem \ref{fond}, the $T_{\phi_x}$ are a unital contractions of $L^2(H_N^{s+})$, and their restrictions to $L^{\infty}(H_N^{s+})$ are NUCP $h$-preserving maps. 
Moreover, Proposition \ref{c04} in the case $N=4$ and Proposition \ref{c05} in the cases $N\ge5$, together with the fact that the $p_{\alpha}$ are finite rank operators, show that for each $x\in I_N$, the operator $T_{\phi_x}$ is compact. To conclude one has to show that for all $x\in I_N$,
\begin{equation}\label{asser}
||T_{\phi_x}a-a||_{L_2}\underset{x\to N}{\longrightarrow}0
\end{equation}
for all $a\in L^{\infty}(H_N^{s+})$ (via $a\in L^{\infty}(H_N^{s+})\hookrightarrow L^2(H_N^{s+})$). First let us prove that it is true for any element $a\in Pol(H_N^{s+})$ i.e. any linear combination of matrix coefficients $U_{ij}^{\alpha}$ of irreducible corepresentations of $C(H_N^{s+})$ (by linearity, we can do that only for the elements $U_{ij}^{\alpha}$). Notice that if $\alpha=a^{l_1}z^{j_1}\dots z^{j_{k_{\alpha}-1}}a^{l_{k_{\alpha}}}$ then $\overline{\alpha}=a^{l_{k_{\alpha}}}z^{-j_{k_{\alpha}-1}}\dots z^{-j_{1}}a^{l_{1}}=a^{l_{k_{\alpha}}}z^{s-j_{k_{\alpha}-1}}\dots z^{s-j_{1}}a^{l_{1}}$. Thus by Proposition \ref{Image}
$
\phi_x(\chi_{\overline{\alpha}})=\psi_x\circ\pi(\chi_{\overline{\alpha}})=\psi_x\circ\pi(\chi_{\alpha})=\phi_x(\chi_{\alpha}).$
Hence, 
\begin{align*}
||T_{\phi_x}U_{ij}^{\alpha}-U_{ij}^{\alpha}||_{L^2}&=||U_{ij}^{\alpha}||_{L_2}\left(1-\prod_{i=1}^{k_{\alpha}}\frac{A_{l_i}(\sqrt{x})}{A_{l_i}(\sqrt{N})}\right),
\end{align*}
so let $x\to N$ and the assertion $(\ref{asser})$ holds for all these matrix coefficient. Now by $L^2$-density of $Pol(H_N^{s+})$ and the fact that all $T_{\phi_x}$, $x\in I_N$, are unital contractions (and thus are uniformly bounded), we obtain that $(\ref{asser})$ is true for any $a\in L^2(H_N^{s+})$.
\end{proof}

\begin{rque}
In \cite{Bic04}, it is proved that there is a $*$-Hopf algebras isomorphism between $C(H_2^{s+})$ and $C^*(\Z_s*\Z_s\times \Z_2)$ (see Example $2.5$ and thereafter in that paper). Furthermore, the Haar state on $C^*(\Z_s)*_wC(S_2^+)$ is given by $h=h_1\otimes h_2$ where $h_2$ is the Haar state on $C(S_2^+)$ and $h_1$ is the free product of the Haar states on $C^*(\Z_s)$). Then, it is clear that $H_2^{s+}$ has the Haagerup property by the stability properties of the Haagerup property on groups (see e.g. \cite{CCJ})

The algebra $C(H_3^{s+})$ is more complicated and does not reduce to a more comprehensive tensor product as for the case $N=2$. We are unable at the moment to prove that $H_3^{s+}$ has the Haagerup property.
\end{rque}

\section*{Acknowledgements}
I am very grateful to my advisors Uwe Franz and Roland Vergnioux for the time they spent discussing the arguments of this paper. I would also like to thank Pierre Fima for discussions on various topics on quantum groups, Amaury Freslon for discussions on averaging methods and approximation properties on quantum groups and Mikael de la Salle for very useful suggestions and commentaries on some special cases of the theorem proved in this paper.

\bibliography{HaagAshJ}

\providecommand{\bysame}{\leavevmode\hbox to3em{\hrulefill}\thinspace}
\providecommand{\MR}{\relax\ifhmode\unskip\space\fi MR }
\providecommand{\MRhref}[2]{%
  \href{http://www.ams.org/mathscinet-getitem?mr=#1}{#2}
}
\providecommand{\href}[2]{#2}
\begin{thebibliography}{10}

\bibitem{BBCC11}
T.~Banica, S.~T. Belinschi, M.~Capitaine, and B.~Collins, \emph{Free {B}essel
  laws}, Canad. J. Math. \textbf{63} (2011), no.~1, 3--37. \MR{2779129
  (2011m:46121)}

\bibitem{Ban96}
Teodor Banica, \emph{Th\'eorie des repr\'esentations du groupe quantique
  compact libre {${\rm O}(n)$}}, C. R. Acad. Sci. Paris S\'er. I Math.
  \textbf{322} (1996), no.~3, 241--244. \MR{1378260 (97a:46108)}

\bibitem{Ban99}
\bysame, \emph{Symmetries of a generic coaction}, Math. Ann. \textbf{314}
  (1999), no.~4, 763--780. \MR{1709109 (2001g:46146)}

\bibitem{BV09}
Teodor Banica and Roland Vergnioux, \emph{Fusion rules for quantum reflection
  groups}, J. Noncommut. Geom. \textbf{3} (2009), no.~3, 327--359. \MR{2511633
  (2010i:46109)}

\bibitem{Bic04}
Julien Bichon, \emph{Free wreath product by the quantum permutation group},
  Algebr. Represent. Theory \textbf{7} (2004), no.~4, 343--362. \MR{2096666
  (2005j:46043)}

\bibitem{Bra11}
Michael Brannan, \emph{Approximation properties for free orthogonal and free
  unitary quantum groups}, Journal f\"ur die reine und angewandte Mathematik
  \textbf{Volume 2012} (2012), 223--251.

\bibitem{Bra12}
\bysame, \emph{Reduced operator algebras of trace-preserving quantum
  automorphism groups}, preprint arXiv:1202.5020, 2012.

\bibitem{CDCH10}
Indira Chatterji, Cornelia Dru{\c{t}}u, and Fr{\'e}d{\'e}ric Haglund,
  \emph{Kazhdan and {H}aagerup properties from the median viewpoint}, Adv.
  Math. \textbf{225} (2010), no.~2, 882--921. \MR{2671183 (2011g:20059)}

\bibitem{CCJ}
Cherix, Cowling, Jolissaint, Julg, and Valette, \emph{Groups with the
  {H}aagerup property}, Progress in Mathematics, vol. 197, Birkh\"auser Verlag,
  Basel, 2001, Gromov's a-T-menability.

\bibitem{CSV12}
Yves Cornulier, Yves Stalder, and Alain Valette, \emph{Proper actions of wreath
  products and generalizations}, Trans. Amer. Math. Soc. \textbf{364} (2012),
  no.~6, 3159--3184. \MR{2888241}

\bibitem{Fim10}
Pierre Fima, \emph{Kazhdan's property {$T$} for discrete quantum groups},
  Internat. J. Math. \textbf{21} (2010), no.~1, 47--65. \MR{2642986
  (2011f:46089)}

\bibitem{Fre12}
Amaury Freslon, \emph{Examples of weakly amenable discrete quantum groups},
  Journal of Functional Analysis (2013).

\bibitem{Haa79}
Uffe Haagerup, \emph{An example of a nonnuclear {$C^{\ast} $}-algebra, which
  has the metric approximation property}, Invent. Math. \textbf{50} (1978/79),
  no.~3, 279--293. \MR{520930 (80j:46094)}

\bibitem{HK01}
Nigel Higson and Gennadi Kasparov, \emph{{$E$}-theory and {$KK$}-theory for
  groups which act properly and isometrically on {H}ilbert space}, Invent.
  Math. \textbf{144} (2001), no.~1, 23--74. \MR{1821144 (2002k:19005)}

\bibitem{KV00}
Johan Kustermans and Stefaan Vaes, \emph{Locally compact quantum groups}, Ann.
  Sci. \'Ecole Norm. Sup. (4) \textbf{33} (2000), no.~6, 837--934. \MR{1832993
  (2002f:46108)}

\bibitem{Pop06}
Sorin Popa, \emph{On a class of type {${\rm II}_1$} factors with {B}etti
  numbers invariants}, Ann. of Math. (2) \textbf{163} (2006), no.~3, 809--899.
  \MR{2215135 (2006k:46097)}

\bibitem{Riv90}
Theodore~J. Rivlin, \emph{Chebyshev polynomials}, second ed., Pure and Applied
  Mathematics (New York), John Wiley \& Sons Inc., New York, 1990, From
  approximation theory to algebra and number theory. \MR{1060735 (92a:41016)}

\bibitem{VV07}
Stefaan Vaes and Roland Vergnioux, \emph{The boundary of universal discrete
  quantum groups, exactness, and factoriality}, Duke Math. J. \textbf{140}
  (2007), no.~1, 35--84. \MR{2355067 (2010a:46166)}

\bibitem{Wang2}
Shuzhou Wang, \emph{Free products of compact quantum groups,}, Comm. Math.
  Phys. \textbf{167} (1995), no.~3, 671--692.

\bibitem{Wang}
\bysame, \emph{Quantum symmetry groups of finite spaces}, Comm. Math. Phys.
  \textbf{195} (1998), no.~1, 195--211. \MR{1637425 (99h:58014)}

\bibitem{Web12}
Moritz Weber, \emph{On the classification of easy quantum groups}, Advances in
  Mathematics \textbf{245} (2013), 500--533.

\bibitem{Wor87}
S.~L. Woronowicz, \emph{Compact matrix pseudogroups}, Comm. Math. Phys.
  \textbf{111} (1987), no.~4, 613--665. \MR{901157 (88m:46079)}

\bibitem{Wor88}
\bysame, \emph{Tannaka-{K}re\u\i n duality for compact matrix pseudogroups.
  {T}wisted {${\rm SU}(N)$} groups}, Invent. Math. \textbf{93} (1988), no.~1,
  35--76. \MR{943923 (90e:22033)}

\bibitem{Wor95}
\bysame, \emph{Compact quantum groups}, Sym\'etries quantiques ({L}es
  {H}ouches, 1995), North-Holland, Amsterdam, 1998, pp.~845--884. \MR{1616348
  (99m:46164)}

\end{thebibliography}

\end{document}